\documentclass[12pt]{article}
\usepackage{amsmath,amssymb,amsbsy,amsfonts,amsthm,latexsym,amsopn,amstext,
                            amsxtra,euscript,amscd}
\usepackage{booktabs}

\newtheorem{theorem}{Theorem}
\newtheorem{lemma}[theorem]{Lemma}

\newtheorem{thm}[theorem]{Theorem}

\def\\{\cr}
\def\({\left(}
\def\){\right)}
\def\[{\left[}
\def\]{\right]}
\def\<{\langle}
\def\>{\rangle}

\def\cD{\mathcal D}

\def\cI{\mathcal I}

\def\cP{\mathcal P}

\def\Z{\mathbb{Z}}

\def\Q{\mathbb{Q}}

\def\mand{\qquad \mbox{and} \qquad}

\def\notdivides{\mathrel{\kern-3pt\not\!\kern3.5pt\bigm|}}

\begin{document}


\title{On the number of polynomials of bounded height that satisfy Dumas's criterion}

\author{
{\sc Randell Heyman}\\
{Department of Computing, Macquarie University} \\
{Sydney, NSW 2109, Australia}\\
{\tt randell@unsw.edu.au}
}

\date{ }
\maketitle

\begin{abstract}
We study integer coefficient polynomials of fixed degree and maximum height $H$, that are irreducible by Dumas's criterion. We call such polynomials \emph{Dumas polynomials}. We derive upper bounds on the number of Dumas polynomials, as $H \rightarrow \infty$.
We also show that, for a fixed degree, the density of Dumas polynomials in all irreducible integer coefficient polynomials is strictly less than 1.
\end{abstract}


\noindent
\section{Introduction}
The two most well-known polynomial irreducibility criteria based on coefficient primality divisibility are probably the Eisenstein criterion and the Dumas criterion. In the last decade and a half results regarding the density of polynomials that satisfy the Eisenstein criterion have been obtained. In particular,  papers by Dobbs and Johnson \cite{DoJo}, Dubickas \cite{Dub} and the author and Shparlinski \cite{Hey,Hey2}. In this paper we explore densities of polynomials that satisfy Dumas's criterion.  This criterion is a sufficient condition for polynomial irreducibility over $\Z$ (and hence $\Q$). It can be thought of as a generalization of the Eisenstein criterion since the Eisenstein criterion is an easy result of Dumas's criterion. The Dumas's criterion is based on the construction of a Newton diagram. Construction of a Newton diagram is explained in the book of Prasolov \cite{Pra}. The explanation is reproduced in slightly edited form below.

Let $p$ be a fixed prime, and let
\begin{align}\label{f(x)}
 f(x)=\sum_{i=0}^nA_ix^i \in \Z[x]
 \end{align}
  be such that $A_0A_n \neq 0$. Represent the nonzero coefficients of $f$ in the form $A_ip^{\alpha_i}$, where $\gcd{(A_i,p)}=1$. To every nonzero coefficient $A_ip^{\alpha_i}$ we assign a point in the plane with coordinates $(i,\alpha_i)$.

Let $P_0=(0,\alpha_0)$ and $P_1=(i_1,\alpha_1)$, where $i_1$ is the largest integer for which there are no points $(i,\alpha_i)$ below the line $P_0P_1$. Further, let $P_2=(i_2,\alpha_2)$, where $i_2$ is the largest integer for which there are no points $(i,\alpha_i)$ below the line $P_1P_2$, etc. The very last segment is of the form $P_{r-l},P_r$, where $P_r=(n,\alpha_n)$.

If some segments of the broken line $P_0\cdots P_r$ pass through points with integer coefficients, then such points will be also considered as vertices of the broken line. In this way, to the vertices $P_0\cdots P_r$, we add $s \ge 0$ more vertices. The resulting broken line $Q_0\cdots Q_{r+s}$ is called \emph{the Newton diagram of polygon $f$ (with respect to $p$)}. The segments $Q_iQ_{i+1}$ are called \emph{segments} of the Newton diagram.

We can now state Dumas's criterion. Let $f \in \Z[x]$ be such that neither the leading coefficient nor the constant is equal to zero. If the Newton diagram for $f$ with respect to any prime $p$ consists of precisely one segment, i.e. consists of a segment containing no points with integer coordinates other than the end points, then $f$ is irreducible. The proof of Dumas's criterion can be found in the book of Prasolov \cite{Pra}[Theorem~2.2.1]. We sometimes call a polynomial that satisfies Dumas's criterion a \emph{Dumas polynomial}.

For integers $n \ge 2$ and $H \ge 1$, let $\cD_n(H)$ be the number of Dumas polynomials of height at most $H$, that is, satisfying $\max\{|A_0|,\dots,|A_n|\}\le H$. We have already noted that the number of polynomials that satisfy the Eisenstein criterion, calculated by the author and Igor Shparlinski ~ $\cite{Hey}$, provides a lower bound on $\cD_n(H)$. Our main result gives an upper bound.

\begin{thm}\label{firstmain}
We have,
$$\cD_n(H) \le (2H)^{n+1}\tau_n + \begin{cases}
O\(H^2 (\log H)^2\),&\quad \text{if $n=2$};\\
O\(H^n\),&\quad  \text{if $n\ge 3$},
\end{cases}$$
where
$$\tau_n=\begin{cases}
1-\prod_{p~\textrm{prime}}\(1-\frac{1}{p}\)^2\(1+\frac{2}{p}\),&\quad \text{if $n=2$}; \\
1-\frac{1}{\zeta(n-1)},&\quad  \text{if $n \ge 3$}.
\end{cases}$$
\end{thm}

\section{Notation}
Let $f(x)$ be as in \eqref{f(x)}. We will always assume that $A_n \ne 0$. We define the height of the polynomial $f$ as
\begin{equation*}
H(f) = \max_{i=0, \ldots, n} |A_i|.
\end{equation*}
As usual the Riemann zeta function is given by
$$\zeta(s)=\sum_{k=1}^\infty \frac{1}{k^s},$$
for all complex numbers $s$ whose real part is greater than 1.
We also recall that the notation $U=O(V)$ is equivalent to the assertion that the inequality $|U| \le c|V|$ holds for some constant $c>0$.

\section{Preparations}
\begin{lemma}\label{gcd ends}
Fix $n=2$. Suppose that $f(x)$ is as in \eqref{f(x)} with $H(f) \le H$ and $A_{1} \ne 0.$ If $f$ is a Dumas polynomial then
$\gcd(A_j,A_k) \ne 1$  for some distinct $j,k \in \{0,1,2\}$.
\end{lemma}

\begin{proof}
Suppose there exists a polynomial $f$, as described above, with the property that $\gcd(A_j,A_k) = 1$ for some distinct $j,k \in \{0,1,2\}$. If for any prime $p$ we have $p\mid  A_{1}$ then clearly $p \nmid A_0$ and $p \nmid A_2$. So the Newton diagram with respect to $p$ consists of $2$ segments. Thus $f$ is not a Dumas polynomial. On the other hand, if for any prime $p$ we have $p \nmid A_{1}$ then the Newton diagram includes the point $(1,0)$. Thus the Newton diagram consists of more than one segment. So again $f$ is not a Dumas polynomial, completing the proof.
\end{proof}

\begin{lemma}\label{gcd middle}
Fix $n \ge 2$. Suppose $f(x)$ is as shown in \eqref{f(x)} with $H(f) \le H$ and $A_{1}A_{2}\cdots A_{n-1} \ne 0.$ If $f$ is a Dumas polynomial then
$\gcd(A_1,A_2,\ldots,A_{n-1}) \ne 1$.
\end{lemma}

\begin{proof}
Suppose $f$ is as described above with $\gcd(A_1,A_2,\ldots,A_{n-1})=1$. For any prime $p$ we must have $p \nmid A_i$ for some $1 \le i \le n-1$. So the Newton diagram for $f$ with respect to $p$ includes the point $(A_i,0)$. Thus the Newton diagram with respect to any prime $p$ does not consist of a single segment. Therefore $f$ is not a Dumas polynomial.
\end{proof}

\section{Proof of Theorem \ref{firstmain}}
Let $f(x)$ be as in \eqref{f(x)} with $H(f) \le H$.
We prove Theorem \ref{firstmain} for $n=2$ and $n \ge 3$ separately.

We start with the $n=2 $ case. To ease notation we use $\gcd^*(A_0,A_{1},A_{2}) \ne 1$ to mean that $A_0,A_{1}$ and $A_{2}$ are not pairwise coprime. That is, $\gcd(A_0,A_{1})\ne 1$ or $\gcd(A_0,A_{2})\ne 1$ or $\gcd(A_{1},A_{2})\ne 1$. We also use $\gcd_*(A_0,A_{1},A_{2}) = 1$ to mean that  $A_0,A_{1}$ and $A_{2}$ are pairwise coprime. That is, $\gcd(A_0,A_{1})=\gcd(A_0,A_{2})=\gcd(A_{1},A_{2})= 1$.

There are $O(H^{2})$ polynomials with $A_{0}A_{1}A_2=0.$ If we have $A_{0}A_{1}A_2 \ne 0$ then, by Lemma \ref{gcd ends}, the polynomial $f$ can only be a Dumas polynomial if $\gcd^*(A_0,A_1,A_2) \ne 1$. Therefore
\begin{align}\label{2}
\cD_{2}(H) &\le \sum_{\substack{1\le |A_0|,|A_{1}|,|A_{2}| \le H \\\gcd^*(A_0,A_1,A_2) \ne 1}}1+O(H^{2})\notag\\
&= \sum_{\substack{1\le A_0,A_{1},A_{2} \le H \\\gcd^*(A_0,A_1,A_2) \ne 1}}8+O(H^{2})\notag\\
&= (2H)^3-\sum_{\substack{1\le A_0,A_{1},A_{2} \le H \\\gcd_*(A_0,A_1,A_2) = 1}}8+O(H^{2}).
\end{align}
From the paper of T$\acute{\textrm{o}}$th~\cite{Tot} we have
\begin{align*}
\sum_{\substack{1\le |A_0|,|A_{1}|,|A_{2}| \le H \\\gcd_*(A_0,A_{1},A_{2})=1 }}1=H^3\prod_{p~\textrm{prime}} \(1-\frac{1}{p}\)^2\(1+\frac{2}{p}\)+O\(H^2(\log H)^2\),
\end{align*}
from which
\begin{align}\label{toth}
\sum_{\substack{1\le A_0,A_{1},A_{2} \le H \\\gcd_*(A_0,A_{1},A_{2}) = 1}}8&=(2H)^3\prod_{p~\textrm{prime}} \(1-\frac{1}{p}\)^2\(1+\frac{2}{p}\)+O\(H^2(\log H)^2\).
\end{align}
Substituting \eqref{toth} into \eqref{2} completes the proof for the $n=2$ case.

For the $n \ge 3$ case fix a $n\ge 3$.
There are $O(H^{n})$ polynomials for which $A_1A_2\cdots A_{n-1}=0.$ If $A_1A_2\cdots A_{n-1}\ne 0$ then, by Lemma \ref{gcd middle}, the polynomial $f$ can only be a Dumas polynomial if $\gcd(A_1,A_2,\ldots,A_{n-1}) \ne 1$. Therefore
\begin{align}\label{eventhm}
\cD_{n}(H)&\le \sum_{\substack{1\le|A_1|,|A_2|,\ldots,|A_{n-1}| \le H\\\gcd(A_1,A_2,\ldots,A_{n-1}) \ne 1 }}1+O(H^{n}).
\end{align}
We infer from Nymann~\cite{Nym} that
\begin{align}\label{coprime}
\sum_{\substack{1\le |A_1|,|A_2|,\ldots,|A_{n-1}|\le H\\\gcd(A_1,A_2,\ldots,A_{n-1}) \ne 1 }}1&=(2H)^{n-1}\(1-\frac{1}{\zeta(n-1)}\)+O(H^{n-2}).
\end{align}
Substituting \eqref{coprime} into \eqref{eventhm} completes the proof for the $n \ge 3$ case. Thus Theorem \ref{firstmain} is proven.

\section{Comments}
Let $\cP_n(H)$ be the number of polynomials of degree $n$ and maximum height $H$. Let $\cI_n(H)$ be the number of irreducible polynomials of degree $n$ and maximum height $H$. Two results immediately follow from Theorem \ref{firstmain}.

Firstly, we note that $\cP_n(H)=(2H)^{n+1}+O(H^n)$ and infer from Cohen ~\cite[Theorem~1]{Coh} that for $n \ge 2$,
$$\lim_{H \to \infty} \frac{\cI_n(H)}{\cP_n(H)}=1.$$
Thus, for $n\ge 2$,
$$\limsup_{H \rightarrow \infty}\frac{\cD_n(H)}{\cP_n(H)}=\limsup_{H \rightarrow \infty}\frac{\cD_n(H)}{\cI_n(H)}=\tau_n.$$

Secondly, $\tau_n<1$ for all $n\ge 2$, and so for $n \ge 2$,
$$\limsup_{H \rightarrow \infty} \frac{\cD_n(H)}{\cP_n(H)} =\limsup_{H \rightarrow \infty} \frac{\cD_n(H)}{\cI_n(H)}<1.$$

Table 1 shows calculated values of upper bounds on the limit superior of $\cD_n(H)/\cP_n(H)$ as $H$ goes to infinity, for various values of $n$. It also includes limit inferior calculations derived from a paper by the  author and Shparlinski~\cite{Hey}. Specifically, lower bounds on the limit inferior of $\cD_n(H)/\cP_n(H)$ as $H$ goes to infinity, for various values of $n$. All summations are over all primes less than 100,000.

\begin{table}[h]
\begin{center}
    \caption{Some upper and lower bounds on $\cD_n(H)/\cP_n(H)$ as $H \rightarrow \infty$}
   \medskip
\medskip
    \begin{tabular}{ | c | c | c |}
    \hline
    \textrm{$n$} & Lower bound & Upper bound \\ \hline
    $2$ & 0.1677 &0.7133\\ \hline

    $3$ & 0.0556 & 0.3922 \\ \hline

    $4$&0.0224& 0.1681  \\ \hline

    $5$&0.0099& 0.0766 \\ \hline
   $6$& 0.0046 & 0.0357 \\ \hline
    $7$&0.0022 & 0.0181\\ \hline
    $8$&0.0010&0.0079 \\ \hline
    $9$&0.0005 & 0.0049 \\ \hline
    $10$&0.0003&0.0020 \\ \hline
    \end{tabular}
  \end{center}
 \end{table}
\newpage
This prompts the following question. Is it possible to obtain tighter bounds or the exact values of
$$\liminf_{H \to \infty} \frac{\cD_n(H)}{\cP_n(H)}
\mand \limsup_{H \to \infty} \frac{\cD_n(H)}{\cP_n(H)}$$
(they most likely coincide)?

We also note that it is possible to find upper bounds on $$\limsup_{H \rightarrow \infty} \cD_n(H)/\cP_n(H)$$ by directly calculating the number of Dumas polynomials for an arbitrary single segment Newton diagram, and then summing over all possible single segment Newton diagrams. There are substantial problems using the inclusion exclusion principle with this approach; a Dumas polynomial with respect to more than one prime may exhibit a different Newton diagram for each of these primes. Whilst results for degree $n>3$ are obtainable without the inclusion exclusion principle, it has not been possible to find any results that are superior to Theorem \ref{firstmain}.

\section{Acknowledgement}
The author would like to thank Igor Shparlinski for numerous helpful suggestions.

\hrulefill
\newline
\newline
2010 \emph{Mathematical Subject Classification:} Primary 11R09.
\newline
\emph{Keywords:}~Irreducible polynomial, Dumas's criterion, coprimality.
\newline

\noindent{\hrulefill}
\end{document}